\documentclass{amsart}

\usepackage{amsmath}
\usepackage{amsfonts}
\usepackage{amssymb}
\usepackage{amsthm}
\usepackage{graphics}
\usepackage{color}
\usepackage{algorithm}
\usepackage{algorithmic}

\input xy
\xyoption{all}

\newtheorem{theorem}{Theorem}[section]
\newtheorem{lemma}[theorem]{Lemma}

\newtheorem{corollary}[theorem]{Corollary}

\theoremstyle{definition}

\theoremstyle{remark}

\numberwithin{equation}{theorem}

\DeclareMathOperator{\Gal}{Gal}

\DeclareMathOperator{\im}{im}

\DeclareMathOperator{\cores}{cores}

\DeclareMathOperator{\res}{res}

\DeclareMathOperator{\point}{point}

\DeclareMathOperator{\vcd}{vcd}

\DeclareMathOperator{\cd}{cd}

\DeclareMathOperator{\red}{red}

\begin{document}

\title[The Hasse Principle and Zero Cycles of Degree One]{Implications of the Hasse Principle for Zero Cycles of Degree One on Principal Homogeneous Spaces}

\author{Jodi Black}

\address{Department of Mathematics and Computer Science, Emory University, Atlanta, Georgia, 30322}

\email{jablack@emory.edu}

\thanks{The results in this work are from a doctoral dissertation in progress under the direction of R. Parimala whom we sincerely thank for her guidance.\\
MSC 2010 Primary Classification: 11E72\\ MSC 2010 Secondary Classification:11E57}

\date{September 30, 2010}

\begin{abstract}
Let $k$ be a perfect field of virtual cohomological dimension $\leq 2$. Let $G$ be a connected linear algebraic group over $k$ such that $G^{sc}$ satisfies a Hasse principle over $k$. Let $X$ be a principal homogeneous space under $G$ over $k$. We show that if $X$ admits a zero cycle of degree one, then $X$ has a $k$-rational point.
\end{abstract}

\maketitle

\section*{Introduction}

The following question of Serre \cite[pg 192 ]{SerreGC} is open in general. 

\bigskip

\begin{description}
\item[Q] Let $k$ be a field and $G$ a connected linear algebraic group defined over $k$. Let $X$ be a principal homogeneous space under $G$ over $k$. Suppose $X$ admits a zero cycle of degree one, does $X$ have a $k$-rational point?
\end{description}

\bigskip

Let $k$ be a number field, let $V$ be the set of places of $k$ and let $k_v$ denote the completion of $k$ at a place $v$. We say that a connected linear algebraic group $G$ defined over $k$ satisfies a \emph{Hasse principle} over $k$ if the map $H^1(k,G) \to \prod_{v \in V} H^1(k_v,G)$ is injective. Let $V_r$ denote the set of real places of $k$. If $G$ is simply connected, then by a theorem of Kneser, the Hasse principle reduces to injectivity of the maps $H^1(k,G) \to \prod_{v \in V_r} H^1(k_v,G)$. That this result holds is a theorem due to Kneser, Harder and Chernousov \cite{ChernousovHasse} ,\cite{HarderHasse},\cite{KneserHasse}. Sansuc used this Hasse principle to show that {\bf Q} has positive answer for number fields. 


Let $k$ be any field and $\Omega$ the set of orderings of $k$. For $v \in \Omega$ let $k_v$ denote the real closure of $k$ at $v$. We say that a connected linear algebraic group $G$ defined over $k$ satisfies a \emph{Hasse principle} over $k$ if the map $H^1(k,G) \to \prod_{v \in \Omega} H^1(k_v,G)$ is injective. It is a conjecture of Colliot-Th\'el\`ene \cite[pg 652]{BayerPariHasse} that a simply connected semisimple group satisfies a Hasse principle over a perfect field of virtual cohomological dimension $\leq 2$. Bayer and Parimala \cite{BayerPariHasse} have given a proof in the case where $G$ is of classical type, type $F_4$ and type $G_2$.

Our goal in this paper is to extend Sansuc's result by providing a positive answer to {\bf Q} when $k$ is a perfect field of virtual cohomological dimension $\leq 2$ and $G^{sc}$ satisfies a Hasse principle over $k$. More precisely, we prove the following:

\begin{theorem}
Let $k$ be a perfect field of virtual cohomological dimension $\leq 2$. Let $\{ L_i\}_{1 \leq i \leq m}$ be a set of finite field extensions of k such that the greatest common divisor of the degrees of the extensions $[L_i:k]$  is 1. Let $G$ be a connected linear algebraic group over $k$. If $G^{sc}$ satisfies a Hasse principle over $k$, then the canonical map
\[
H^1(k,G) \rightarrow \displaystyle\prod_{i=1}^m H^1(L_i,G)
\]
has trivial kernel.
\end{theorem}

We obtain the following as a corollary: 

\begin{corollary}
Let $k$ be a perfect field of virtual cohomological dimension $\leq 2$. Let $\{ L_i\}_{1 \leq i \leq m}$ be a set of finite field extensions of k such that the greatest common divisor of the degrees of the extensions $[L_i:k]$  is 1. Let $G$ be a connected linear algebraic group over $k$. If the simple factors of $G^{sc}$ are of classical type, type $F_4$ or type $G_2$ then the canonical map
\[
H^1(k,G) \rightarrow \displaystyle\prod_{i=1}^m H^1(L_i,G)
\]
is injective.
\end{corollary}

Sansuc's proof of a positive answer to {\bf Q} over number fields relies on the surjectivity of the map $H^1(k,\mu) \to \prod_{v \in V_r}H^1(k_v,\mu)$ for $\mu$ a finite commutative group scheme. This result is a consequence of the Chebotarev density theorem and does not extend to a general field of virtual cohomological dimension $\leq 2$. Even in the case $\mu = \mu_2$, the surjectivity of the map $H^1(k, \mu) \to \prod_{v \in \Omega} H^1(k_v, \mu)$ imposes severe conditions on $k$ like the SAP property. The main content of this paper is to replace the arithmetic in Sansuc's paper with a norm principle over a real closed field. 

\section{Algebraic Groups}

In this section, we review some well-known facts from the theory of algebraic groups and define some notation used in the remainder of the work. 

Let $k$ be a field. An \emph{algebraic group} $G$ over $k$ is a smooth group scheme of finite type. A surjective morphism of algebraic groups with finite kernel is called an \emph{isogeny} of algebraic groups.
An isogeny $ G_1 \to G_2$ is said to be \emph{central} if its kernel is a central subgroup of $G_1$.

An \emph{algebraic torus} is an algebraic group $T$ such that $T(\bar k)$ is isomorphic to a product of multiplicative groups $G_{m, \bar k}$. A torus $T$ is said to be \emph{quasitrivial} if it is a product of groups of the form $R_{E_i/k}G_m$ where $\{E_i\}_{1 \leq i \leq r}$ is a family of finite field extensions of $k$. 

An algebraic group $G$ is called \emph{linear} if it is isomorphic to a closed subgroup of $GL_n$ form some $n$, or equivalently, if its underlying algebraic variety is affine. Of particular interest among connected linear algebraic groups are semisimple groups and reductive groups. 

A connected linear algebraic group is called \emph{semisimple} if it has no nontrivial, connected, solvable, normal subgroups. A semisimple group $G$ is said to be \emph{simply connected} if every central isogeny $G^\prime \to G$ is an isomorphism. We can associate to any semisimple group a simply connected group $\tilde G$ (unique up to isomorphism) such that there is a central isogeny $\tilde G \to G$. We refer to $\tilde G$ as the \emph{simply connected cover} of $G$.

Any simply connected semisimple group is a product of simply connected simple algebraic groups \cite[Theorem 26.8]{KMRT}. Any simple algebraic group belongs to one of four infinite families $A_n$, $B_n$, $C_n$, $D_n$ or is of type $E_6,E_7,E_8,F_4$ or $G_2$ (see for example \cite[\S 26]{KMRT}). A simple group which is of type $A_n, B_n, C_n$ or $D_n$ but not of type trialitarian $D_4$ is said to be a \emph{classical group}. All other simple groups are called \emph{exceptional groups}.

A connected linear algebraic group is called  \emph{reductive} if it has no nontrivial, connected, unipotent, normal subgroups. Given a connected linear algebraic group $G$, the \emph{unipotent radical} of $G$ denoted $G^u$ is the maximal connected unipotent normal subgroup of $G$. It is clear that $G/G^u$ is always a reductive group. We denote $G/G^u$ by $G^{\red}$. The commutator subgroup of $G^{\red}$ is a semisimple group which we denote $G^{ss}$. We denote the simply connected cover of $G^{ss}$ by $G^{sc}$.

A \emph{special covering} of a reductive group $G$ is an isogeny
\[
1 \to \mu \to G_0 \times S \to G \to 1
\]

\noindent
where $G_0$ is a simply connected semisimple algebraic $k$-group and $S$ is a quasitrival $k$-torus. Given a reductive group $G$ there exists an integer $n$ and a quasitrival torus $T$ such that $G^n \times T$ admits a special covering \cite[Lemme 1.10]{Sansuc}.

\section{Galois Cohomology and Zero Cycles}

For our convenience, we will discuss {\bf Q} in the context of Galois Cohomology. We briefly review some of the notions from Galois Cohomology we will use and then restate {\bf Q} in this setting. 

Let $k$ be a field and $\Gamma_k = \Gal(\bar k/k)$ be the absolute Galois group of $k$. For an algebraic $k$-group $G$, let $H^i(k, G) = H^i(\Gamma_k, G(\bar k))$ denote the Galois Cohomology of $G$ with the assumption $i \leq 1$ if $G$ is not abelian. For any $k$-group $G$, $H^0(k, G)=G(k)$ and $H^1(k,G)$ is a pointed set which classifies the isomorphism classes of principal homogeneous spaces under $G$ over $k$. The point in $H^1(k,G)$ corresponds to the principal homogeneous space with rational point. We will interchangeably denote the point in $H^1(k,G)$ by \emph{point} or 1. 

Each $\Gamma_k$-homomorphism $f:G \to G^{\prime}$ induces a functorial map $H^i(k, G) \to H^i(k, G^{\prime})$ which we shall also denote by $f$. Given an exact sequence of $k$-groups,
\[
\xymatrix{
1 \ar[r] &G_1 \ar[r]^-{f_1} &G_2 \ar[r]^-{f_2} &G_3 \to 1
}
\]

\noindent
there exists a connecting map $\delta_0:G_3(k) \to H^1(k, G_1)$ such that the following is an exact sequence of pointed sets.
\[
\xymatrix{
G_1(k) \ar[r]^-{f_1} &G_2(k) \ar[r]^-{f_2} &G_3(k) \ar[r]^-{\delta_0} &H^1(k,G_1) \ar[r]^-{f_1} &H^1(k,G_2) \ar[r]^-{f_2} &H^1(k,G_3)}
\]

\noindent
If $G_1$ is central in $G_2$, there is in addition a connecting map $\delta_1:H^1(k, G_3) \to H^2(k, G_1)$ such that the following is an exact sequence of pointed sets. 
\[
\xymatrix{
G_3(k) \ar[r]^-{\delta_0} &H^1(k,G_1) \ar[r]^-{f_1} &H^1(k,G_2) \ar[r]^-{f_2} &H^1(k,G_3) \ar[r]^-{\delta_1} &H^2(k,G_1)
}
\]

Given a field extension $L$ of $k$, $\Gal(\bar k/ L) \subset \Gal(\bar k/k)$ and there is a restriction homomorphism $\res: H^1(k,G) \to H^1(L,G)$. If $G$ is a commutative group, and if the degree of $L$ over $k$ is finite, there is also a corestriction homomorphism $\cores:  H^1(L,G) \to H^1(k,G)$. The composition $\cores \circ \res$ is multiplication by the degree of $L$ over $k$.

Let $p$ be any prime number. The \emph{$p$-cohomological dimension} of $k$ is less than or equal to $r$  (written $\cd_p(k) \leq r$) if $H^n(k, A) =0$ for every $p$-primary torsion $\Gamma_k$-module $A$ and $n>r$. The \emph{cohomological dimension} of $k$ is less than or equal to $r$ (written $\cd(k) \leq r$), if $\cd_p(k) \leq r$ for all primes $p$. Finally, the \emph{virtual cohomological dimension} of $k$, written $\vcd(k)$ is precisely the cohomological dimension of $k(\sqrt{-1})$. If $k$ is a field of positive characteristic then $\vcd(k)= \cd(k)$.
 
Let $X$ be a scheme. For any closed point $x \in X$, let $\mathcal{O}_x$ be the local ring at $x$ and let $\mathfrak{M}_x$ be its maximal ideal. The \emph{residue field} of $x$ written $k(x)$ is $\mathcal{O}_x/\mathfrak{M}_x$. \emph{Zero cycles} of $X$ are elements of the free abelian group on closed points $x \in X$. We may associate to any zero cycle $\sum n_ix_i$ on $X$ its \emph{degree} $\sum n_i[k(x_i):k]$ where $k(x_i)$ is the residue field of $x_i$.

A closed point with residue field $k$ is called a \emph{rational point}. It is clear that if $x$ is a closed point of a variety $X$ over $k$ then it is a rational point of $X_{k(x)}$. We have seen that the point in $H^1(*, G)$ is the principal homogeneous space under $G$ over $*$ with a rational point. Therefore, a principal homogeneous space $X$ under $G$ over $k$, with zero cycle $\sum n_i x_i$ is an element of the kernel of the product of the restriction maps $H^1(k,G) \to \prod H^1(k(x_i),G)$. If the zero cycle is of degree one, then the field extensions $k(x_i)$ are necessarily of coprime degree over $k$. 

Guided by this insight, one may restate {\bf Q} as follows.

\begin{description}
\item[Q] Let $k$ be a field and let $G$ be a connected, linear algebraic group defined over $k$. Let $\{ L_i \}_{1 \leq i \leq m}$ be a collection of finite extensions of $k$ with $\text{gcd}([L_i:k])=1$. Does the canonical map
\[
H^1(k,G) \to \displaystyle \prod_{i=1}^m H^1(L_i,G)
\]
have trivial kernel?

\end{description}

\section{Orderings of a Field}

We recall some basic properties of orderings of a field \cite{Scharlau}.

An \emph{ordering} $\nu$ of a field $k$ is given by a binary relation $\leq_\nu$ such that for all $a,b,c \in k$.

\begin{itemize}
\item $a \leq_v a$
\item If $a \leq_v b$ and $b \leq_\nu c$ then $a\leq_v c$
\item If $a\leq_v b$ and $b\leq_v a$ then $a=b$
\item Either $a \leq_v b$ or $b \leq_v a$
\item If $a \leq_v b$ then $a+c \leq_v b +c$
\item If $a \leq_v b$ and $0 \leq_v c$ then $ca \leq_v cb$
\end{itemize}

A field $k$ which admits an ordering is necessarily of characteristic 0. If $k$ is a field with an ordering $v$, an \emph{algebraic extension} of the ordered field $(k,v)$ is an algebraic field extension $L$ of $k$ together with an ordering $v^\prime$ on $L$ such that $v^\prime$ restricted to $k$ is $v$. If $L$ is a finite field extension of $k$ of odd degree there is always an algebraic extension $(L,v^\prime)$ of $(k,v)$ \cite[Chapter 3, Theorem 1.10]{Scharlau}.

A field $k$ is said to be \emph{formally real} if -1 is not a sum of squares in $k$. A field $k$ is called a \emph{real closed field} if it is a formally real field and no proper algebraic extension is formally real. There is a unique ordering  $\square$ on a real closed field. This ordering is defined by the relation $a \leq b$ if and only if $b-a$ is a square in $k$. Further, if $k$ is a real closed field, then $k(\sqrt{-1})$ is algebraically closed \cite[Theorem 2.3 (iii)]{Scharlau}. 


If $L$ is a finite field extension of $k$, then $k_v \otimes L$ is isomorphic to a product of the form $\prod k_v \prod k_v(\sqrt{-1})$. Also, since $k_v(\sqrt{-1})$ is an algebraic closure for $k$ there is a natural inclusion $\Gal(\bar k, k_v) \subset \Gamma_k$ and thus a restriction map $H^1(k,G) \to H^1(k_v,G)$.

\section{Main Result}

\noindent
In the discussion which follows we will need the following lemmas.

\begin{lemma} \label{sc}
Let $k$ be a field and let $G$ be a reductive group over $k$. Fix an integer $n$ and a quasitrivial torus $T$ such that $G^n \times T$ admits a special covering
\[
1 \to \mu \to G_0 \times S \to G^n \times T \to 1
\]
Then $G^{sc}$ satisfies a Hasse Principle over $k$ if and only if $G_0$ satisfies a Hasse principle over $k$.
\end{lemma}

\begin{proof}

Taking commutator subgroups we have a short exact sequence

\[
1 \to \tilde \mu \to [G_0 \times S: G_0 \times S] \to [G^n \times T:G^n \times T] \to 1
\]

Since $S$ and $T$ are tori, $[G_0 \times S: G_0 \times S] \cong [G_0:G_0]$ and $[G^n \times T:G^n \times T] = [G^n:G^n]$. That $G_0$ is semisimple gives $[G_0:G_0] = G_0$. It is clear that $[G^n:G^n] = [G:G]^n$ which in turn is $(G^{ss})^n$ by definition of $G^{ss}$. Therefore, we have the following short exact sequence
\[
1 \to \tilde \mu \to G_0 \to (G^{ss})^n \to 1
\]

\noindent
where $\tilde \mu$ is some finite group schme. In particular, $G_0$ is a simply connected cover of $(G^{ss})^n$. Since $(G^{sc})^n$ is certainly a simply connected cover of $(G^{ss})^n$, uniqueness of the simply connected cover of $(G^{ss})^n$ gives $(G^{sc})^n \cong G_0$. In particular, the simple factors of $G^{sc}$ are the same as the simple factors of $G_0$ and $G^{sc}$ satisfies the Hasse principle over $k$ if and only if $G_0$ satisfies the Hasse principle over $k$.
\end{proof}

\begin{lemma} \label{NormPrinciple}
Let $k$ be a real closed field and let $G$ be a reductive group over $k$ which admits a special covering 
\begin{equation}
1 \to \mu \to G_0 \times S \to G \to 1
\end{equation}

Let $L$ be a finite \`etale $k$-algebra. Let $\delta$ be the first connecting map in Galois Cohomology and let $N_{L/k}$ denote the corestriction map $H^1(k \otimes L, \mu) \to H^1(k, \mu)$. Then
\[
\xymatrix{
 N_{L/k} (\im(G(k \otimes L) \ar[r]^-{\delta_L} &H^1(k \otimes L, \mu))  \subset \im(G(k) \ar[r]^-{\delta} &H^1(k, \mu))
 }
 \]
\end{lemma}

\begin{proof}
Since $k$ is real closed, there exists finite numbers $r$ and $s$ such that $k \otimes L$ is isomorphic to a product of $r$ copies of $k$ and $s$ copies of  $k(\sqrt{-1})$. Thus 
 \[
H^1(k \otimes L, \mu) \cong \prod_{r \, \text{copies}} H^1(k, \mu) \prod_{s \, \text{copies}} H^1(k(\sqrt{-1}),\mu)
\]

Since $k$ is real closed, $k(\sqrt{-1})$ is algebraically closed, $H^1(k(\sqrt{-1}), \mu)$ is trivial and $H^1(k \otimes L, \mu)$ is just a product of $r$ copies of $H^1(k, \mu)$. Therefore, \[N_{L/k}:H^1(k \otimes L, \mu) \to H^1(k, \mu)\] is just the product map \[\prod_{r \, \text{copies}} H^1(k, \mu) \to H^1(k, \mu)\] That $k \otimes L$ is a product of $r$ copies of $k$ and $s$ copies of $k(\sqrt{-1})$ also gives that
\[
G(k \otimes L) \cong \prod_{r \, \text{copies}} G(k) \prod_{s \, \text{copies}} G(k(\sqrt{-1}))
\] 
Therefore, the connecting map \[\xymatrix{ \displaystyle\prod_{r \, \text{copies}} G(k) \prod_{s \, \text{copies}}G(k(\sqrt{-1})) \ar[r]^-{\delta} &\displaystyle\prod_{r \, \text{copies}} H^1(k, \mu) \prod_{s \, \text{copies}} H^1(k(\sqrt{-1}), \mu)}\] is just the product of the connecting maps \[
G(k) \to H^1(k, \mu)
\] and \[
G(k(\sqrt{-1}) \to H^1(k(\sqrt{-1}), \mu)
\]
\noindent
the latter of which is necessarily the trivial map. 

So choose \[(x_1,\ldots, x_r, y_1, \ldots, y_s ) \in G(k \otimes L)\] Then
\begin{eqnarray*}
N_{L/k}(\delta(x_1,\ldots, x_r, y_1, \ldots, y_s)) &=&N_{L/k}(\delta(x_1), \ldots, \delta(x_r), \delta(y_1), \ldots, \delta(y_s))\\
&=& \delta(x_1)\cdots \delta(x_r)\\
&=& \delta(x_1\cdots x_r)
\end{eqnarray*}

\noindent
Since the $x_i$ were chosen to be in $G(k)$ for all $i$, then $x_1\cdots x_r \in G(k)$ and the desired result holds.
\end{proof}

\begin{lemma} \label{lemma1}
 Let $G$ be a reductive group and $L$ be a finite field extension of $k$ of odd degree. The kernel of the canonical map $H^1(k,G) \to H^1(L,G)$ is contained in the kernel of the canonical map $H^1(k,G) \to \prod_{v \in \Omega} H^1(k_v,G)$.
\end{lemma}

\begin{proof}
By \cite[Chapter 3, Theorem 1.10]{Scharlau} each ordering $v$ of $k$ extends to an ordering $w$ of $L$, in particular each real closure $k_v$ is $L_w$ for some ordering $w$ on $L$. Since the natural map $H^1(k,G) \to H^1(L_w,G)$ factors through the canonical map $H^1(k,G) \to H^1(L,G)$, the desired result is immediate.
\end{proof}

\noindent
We now return to the result which is the main goal of this paper. 

\begin{theorem} \label{meta}
Let $k$ be a perfect field of virtual cohomological dimension $\leq 2$ and let $G$ be a connected linear algebraic group over $k$. Let $\{ L_i\}_{1 \leq i \leq m}$ be a set of finite field extensions of k such that the greatest common divisor of the degrees of the extensions $[L_i:k]$  is 1. If $G^{sc}$ satisfies a Hasse principle over $k$, then the canonical map
\[
H^1(k,G) \rightarrow \displaystyle\prod_{i=1}^m H^1(L_i,G)
\]
has trivial kernel.
\end{theorem}

\begin{proof}

By definition of the groups involved, the following sequence is exact 
\begin{equation} \label{short}
1 \to G^u \to G \to G^{\red} \to 1
\end{equation}

\noindent
Since $G^u$ is unipotent, $H^i(k,G^u)$ is trivial for $i \geq 1$ and \eqref{short} induces the long exact sequence in Galois Cohomology
\[
1 \to H^1(k,G) \to H^1(k, G^{\red}) \to 1
\]
which gives that $H^1(k,G) \cong H^1(k, G^{\red})$. Thus to prove \ref{meta} it is sufficient to consider the case where $G$ is a reductive group. Then fix an integer $n$ and quasitrivial torus $T$ such that $G^n \times T$ admits a special covering
\[
1 \to \mu \to G_0 \times S \to G^n \times T \to 1
\]
By functoriality, $H^1(k, G^n \times T) \cong H^1(k,G)^n \times H^1(k,T)$ and since $T$ is quasitrivial, $H^1(k,T) = 1$. It follows that our result holds for $G$ if and only if it holds for $G^n \times T$. Replacing $G$ by $G=G^n \times T$ we assume that $G$ admits a special covering
\[
1 \to \mu \to G_0 \times S \to G \to 1
\]

If $k$ is a field of positive characteristic, $\cd(d)=\vcd(k)=2$. Since $k$ has no orderings and by hypothesis $G^{sc}$ satisfies a Hasse principle over $k$ then $H^1(k,G^{sc})= \{1\}$. In particular $H^1(k,G_0) = \{1\}$ and the special covering of $G$ above induces the following commutative diagram with exact rows
\begin{equation}\xymatrix{
1 \ar[r] &H^1(k,G) \ar[r]^-h \ar[d]^-q &H^2(k, \mu) \ar[d]^-r\\
1 \ar[r] &\prod_iH^1(L_i,G) \ar[r] &\prod_iH^2(L_i, \mu)\\
}
\end{equation}
Choose $\lambda \in \ker(q)$. By commutativity of the diagram $h(\lambda) \in \ker(r)$. A restriction-corestriction argument gives $r$ has trivial kernel. Thus $h(\lambda) =$ point. Then by exactness of the top row of the diagram, $\lambda =$ point. (c.f. \cite{CTBorovoi} for the case $k$ a \lq\lq good\rq\rq \,field of cohomological dimension 2.)

Therefore, we may assume that the characteristic of $k$ is zero. Fix an index $i$. The  special covering of $G$ above induces the following commutative diagram with exact rows where the vertical maps are the restriction maps.
\begin{equation}\label{meta1}
\xymatrix{
H^1(k, \mu) \ar[r]^-f \ar[d] &H^1(k, G_0) \ar[r]^-g \ar[d]^-p &H^1(k,G) \ar[r]^-h \ar[d]^-q &H^2(k, \mu) \ar[d]^-r\\
\prod H^1(L_i, \mu) \ar[r]^-f & \prod H^1(L_i, G_0) \ar[r]^-g & \prod H^1(L_i, G) \ar[r] & \prod H^2(L_i, \mu)
}
\end{equation}

Let $\lambda$ be in $\ker(q)$. Taking $\cores \circ \res$ we find that $r$ has trivial kernel and thus by commutativity of \eqref{meta1}, $\lambda$ is in $\ker(h)$. By exactness of the top row, we choose $\lambda^\prime \in H^1(k,G_0)$ such that $g(\lambda^\prime)=\lambda$. Write $p(\lambda^\prime) = (\lambda^\prime_{L_i})$. Since $g(\lambda^\prime_{L_i}) = \point$, by exactness of the bottom row of \eqref{meta1} choose $\eta_{L_i} \in H^1(L_i, \mu)$ such that $f(\eta_{L_i})= \lambda^{\prime}_{L_i}$. 

For each ordering $v$ of $k$, the special covering of $G$ above also induces the following commutative diagram with exact rows.
\begin{equation}\label{vcd3}
\xymatrix{
H^1(k, \mu) \ar[r]^-f \ar[d] &H^1(k, G_0) \ar[r]^-g \ar[d]^-{p^{\prime}} &H^1(k,G) \ar[r]^-h \ar[d]^-{q^{\prime}} &H^2(k, \mu) \ar[d]^-{r^{\prime}}\\
 \prod_{v \in \Omega} H^1(k_v, \mu) \ar[r]^-f &\prod_{v \in \Omega} H^1(k_v, G_0) \ar[r]^-g &\prod_{v \in \Omega} H^1(k_v, G) \ar[r] &\prod_{v \in \Omega} H^2(k_v, \mu)
}
\end{equation}
By Lemma \ref{lemma1}, $\lambda$ is in the kernel of $q^\prime$. Thus by commutativity of \eqref{vcd3}, $(\lambda^\prime_v)=p^\prime(\lambda^\prime)$ is in $\ker(g)$. Then by exactness of the bottom row of \eqref{vcd3} choose $\alpha_v \in H^1(k_v, \mu)$ such that $f (\alpha_v) =\lambda^{\prime}_v$. Let $(\alpha_v)_{L_i}$ denote the image of $\alpha_v$ under the canonical map $ H^1(k_v,\mu) \to H^1(k_v  \otimes L_i, \mu)$. Let $(\eta_{L_i})_v$ denote the image of $\eta_{L_i}$ under the canonical map $H^1(L_i,\mu) \to H^1(k_v  \otimes L_i, \mu)$. 

By choice of $\alpha_v$ and $\eta_{L_i}$, $f((\alpha_v)_{L_i})= (\lambda^\prime_v)_{L_i} = (\lambda^\prime_{L_i})_v = f((\eta_{L_i})_v)$. In particular, $f((\alpha_v)_{L_i}(\eta_{L_i})_v^{-1})$ is the point in $H^1(k_v \otimes L_i, G_0)$. We have a commutative diagram
\begin{equation}\label{meta2}
\xymatrix{
G(k_v) \ar[r]^-{\delta} \ar[d] &H^1(k_v, \mu) \ar[r]^-f \ar[d]  &H^1(k_v, G_0) \ar[d]\\
\prod _i G(k_v \otimes L_i) \ar[r]^-{\delta_{L_i}} &\prod_i H^1(k_v \otimes L_i,\mu) \ar[r]^-f &\prod_i H^1(k_v \otimes L_i, G_0)
}
\end{equation}
Exactness of the bottom row of \eqref{meta2} gives that $(\alpha_v)_{L_i}(\eta_{L_i})_v^{-1}$ is in the image of $\delta_{L_i}$. Choose $m_i$ such that $\sum m_i[L_i:k] =1$. Since $\delta_{L_i}$ is multiplicative, it follows that for each index $i$, $(\alpha_v)^{m_i}_{L_i}((\eta_{L_i})_v^{-1})^{m_i}$ is in the image of $\delta_{L_i}$

By Lemma \ref{NormPrinciple} above, there exists $\gamma_v$ in $G(k_v)$ such that \[
\delta(\gamma_v) = \prod_i N_{L_i/k}((\alpha_v)^{m_i}_{L_i}((\eta_{L_i})_v^{-1})^{m_i})\] Since by restriction-corestriction $ N_{L_i/k}((\alpha_v)^{m_i}_{L_i}) = \alpha_v^{m_i[L_i:k]}$. It follows that 

\begin{eqnarray*} \delta(\gamma_v) & = &\prod_i N_{L_i/k}((\alpha_v)^{m_i}_{L_i}((\eta_{L_i})_v^{-1})^{m_i})\\
& = & \alpha_v^{\sum_i m_i[L_i:k]}\prod_i (N_{L_i/k}(\eta_{L_i})_v^{-1})^{m_i}\\ & = &\alpha_v \prod_i (N_{L_i/k}(\eta_{L_i})_v^{-1})^{m_i} 
 \end{eqnarray*}

\noindent
In turn

\[ \delta(\gamma_v) \prod_i (N_{L_i/k}(\eta_{L_i})_v)^{m_i} = \alpha_v
 \]

\noindent
Since $f$ is well-defined on the cosets of $G(k_v)$ in $H^1(k_v,\mu)$ \cite{MerkurjevNP} and the top row of \eqref{meta2} is exact, it follows that 
\[  f\left(\prod_i(N_{L_i/k}(\eta_{L_i})_v)^{m_i}\right) = f(\alpha_v) 
\]
By choice of $\alpha_v$ the latter is $\lambda^\prime_v$. Since $G^{sc}$ satisfies a Hasse principle over $k$, Lemma \ref{sc} gives that $G_0$ satisfies a Hasse principle over $k$. In particular, the map $H^1(k, G_0) \to \prod_v H^1(k_v, G_0)$ is injective, and since 
$f(\prod_i(N_{L_i/k}(\eta_{L_i})^{m_i}))_v = \lambda^\prime_v$ for all $v$, then 

\[f \left(\prod _i(N_{L_i/k}(\eta_{L_i}))^{m_i}\right) = \lambda^\prime
\]
Taking $g$ as in \eqref{meta1} above 
\[
g\left(f \left(\prod _i(N_{L_i/k}(\eta_{L_i}))^{m_i}\right)\right) = g(\lambda^\prime)
\]
Then by exactness of the top row of \eqref{meta1}, $\lambda =g(\lambda^{\prime})= $point.
\end{proof}

Applying \cite[Theorem 10.1]{BayerPariHasse} a Serre twist we obtain the following corollary:

\begin{corollary}
Let $k$ be a perfect field of virtual cohomological dimension $\leq 2$. Let $\{ L_i\}_{1 \leq i \leq m}$ be a set of finite field extensions of k such that the greatest common divisor of the degrees of the extensions $[L_i:k]$  is 1. Let $G$ be a connected linear algebraic group over $k$. If the simple factors of $G^{sc}$ are of classical type, type $F_4$ or type $G_2$ then the canonical map
\[
H^1(k,G) \rightarrow \displaystyle\prod_{i=1}^m H^1(L_i,G)
\]
is injective.
\end{corollary}

\bibliographystyle{amsplain}
\bibliography{paper2}
\end{document}